\documentclass[leqno,11pt]{amsart} %leqno is the option to put formula numbers on the left side
\usepackage{amssymb}
\usepackage{bm}
\usepackage[pdftex]{graphicx}
\theoremstyle{plain} %text of this environment is typesetted in italics
\newtheorem{theorem}{\indent\sc Theorem}[section]

\newtheorem{corollary}[theorem]{\indent\sc Corollary}

\newtheorem{fact}[theorem]{\indent\sc Fact}
\theoremstyle{definition} %text of this environment is typesetted in roman letters

\newtheorem{remark}[theorem]{\indent\sc Remark}
\newtheorem{example}[theorem]{\indent\sc Example}

\def\C{{\mathbb{C}}}%   \C == \mathbf{C}
\def\R{{\mathbb{R}}}%   \R == \mathbb{R}
\def\D{{\mathbb{D}}}%  \D == \mathbb{D}
\def\L{{\mathbb{L}}}%  \L == \mathbb{L}
\def\Pi{{\mathbf{P}}}%  \Pi == \mathbf{P}
\def\E{{\mathbb{E}}}%   \E == \mathbb{E}
\begin{document}

\title[Spacelike CMC surfaces in Lorentz-Minkowski space]{A note on isothermic coordinate systems for \\ spacelike surfaces with constant mean curvature \\ in Lorentz-Minkowski space} %title of paper and the running head option

\author[Y.~Kawakami]{Yu Kawakami} %first author's name and the running head option

\author[K.~Satake]{Kaito Satake} %second author's name and the running head option

\dedicatory{Dedicated to Professor Wayne Rossman on his sixtieth birthday}

%%%%%%%%%%%%%%% footnote %%%%%%%%%%%%%%%%
%\If you use the latest version of amsart.cls, write as follows:
%\subjclass[2020]{ %2020 MSC numbers
%Primary 00; Secondary 00.
%}

\renewcommand{\thefootnote}{\fnsymbol{footnote}}
\footnote[0]{2020\textit{ Mathematics Subject Classification}.
Primary 53A10; Secondary 53C24, 53C42.}
%\date{January 22, 2026.}
\thanks{%acknowledgment of support etc. if any
This work was supported by JSPS KAKENHI Grant Number JP23K03086. 
}
%%%%%%%%%%%% Authors addresses %%%%%%%%%%%%%

\address{% Second Author
Faculty of Mathematics and Physics, \endgraf
Kanazawa University \endgraf
Kanazawa, 920-1192, \endgraf
Japan
}
\email{y-kwkami@se.kanazawa-u.ac.jp}

\address{% first Author
Nagano Prefectural Akaho high school, \endgraf
Komagane, 399-4117, \endgraf
Japan
}
\email{azuozhu96@gmail.com}

%%%%%%%%%%%%%%%%%%%%%%%%%%%%%%%%%%%%%%%%%

\maketitle

\begin{abstract}
In this note, we use isothermic coordinate systems to explore global properties of space-like surfaces with constant mean curvature in the Lorentz-Minkowski three-space. 
\end{abstract}

\section{Introduction}\label{sec1}
The existence of special coordinate systems often plays a crucial role in proving global properties in surface theory. A well-known example is the proof of the Hilbert theorem by using a Tschebyscheff coordinate system, and it states that a surface in the Euclidean $3$-space $\E^3$ with constant negative Gaussian curvature must have singular points. For further details, see \cite[Chapter IX]{Ho1983}. As another example, in \cite[Chapter 5]{Os1986}, a proof of the Bernstein theorem by using a special isothermal coordinate system is described, and it asserts that any entire minimal graph in $\E^3$ must be a plane. Recently, Hasanis, Savas-Halilaj, Vlachos \cite{HHV2011} characterized complex analytic curves within the class of entire minimal graphs in the Euclidean $4$-space $\E^4$ via their Jacobians also by using this coordinate system.

First we prove the existence (Theorem \ref{thm-ICC}) of isothermic coordinate systems for space-like surfaces in the Lorentz-Minkowski $3$-space $\L^3$ with non-zero constant mean curvature (abbrev. CMC). Theorem \ref{thm-ICC} is already a known result, see for example, \cite{BP1999}, \cite[Proposition 1.4]{Ino1997}, \cite{HJ2003}. We then demonstrate its application to global properties. Space-like surfaces in $\L^3$ with non-zero CMC have interesting global properties. See \cite{HKKT2021, Sa2008, Tr1982} for example. Recently, Dorfmeister, Inoguchi and Kobayashi \cite{DIK2016} showed the relationship between minimal surfaces in the $3$-dimensional Heisenberg group and space-like surfaces in $\L^3$ with non-zero CMC by the Sym-formula. 

We here provide a much simpler proof of some uniqueness theorems (Corollary \ref{cor-uni}, Theorem \ref{thm-KY}) by using the isothermic coordinate systems for complete space-like surfaces in $\L^3$ with non-zero CMC, as previously proved by Milnor \cite{Mi1983} and Yamada \cite{Yam1988}.  
    
Finally, the authors gratefully thank the useful advice from Shintaro Akamine, Atsufumi Honda  
and warm encouragement from Kazuki Takezawa during the preparation of this note. 
The authors also acknowledge the reviewer for a careful reading and valuable comments.  

%%%%%%%%%%%%%

\section{Preliminaries}\label{sec2} 
In this section, we briefly review fundamental notions of space-like surfaces in $\L^3$, analogous to those in the theory of surfaces in $\E^3$, using notation consistent with that of \cite{Ke2003}. 

\subsection{Basic theory}\label{sec2-1}
Let $\L^3 = (\R^3, \langle \,,\, \rangle_{L})$ be the Lorentz-Minkowski $3$-space with Lorentz metric   
\[
\langle \bm{a}, \bm{b} \rangle_{L} := a_{1}b_{1} + a_{2}b_{2} -a_{3}b_{3},  
\]
and vector product   
\[
\bm{a} \times_{L} \bm{b} := (a_{2}b_{3}-a_{3}b_{2}, a_{3}b_{1}-a_{1}b_{3}, -(a_{1}b_{2}-a_{2}b_{1})), 
\]
where $\bm{a}=(a_{1}, a_{2}, a_{3}), \bm{b}=(b_{1}, b_{2}, b_{3}) \in \R^3$. We can easily show that 
\[
\langle \bm{a} \times_{L} \bm{b}, \bm{a} \rangle_{L} = \langle \bm{a} \times_{L} \bm{b}, \bm{b} \rangle_{L} = 0 \quad \text{and} \quad \langle \bm{a} \times_{L} \bm{b}, \bm{c} \rangle_{L} 
= \det(\bm{a}, \bm{b}, \bm{c}), 
\]
where $\bm{a}, \bm{b}, \bm{c} \in \L^3$. Here, $\bm{a}, \bm{b}, \bm{c}$ are considered as column vectors and $\det(\bm{a}, \bm{b}, \bm{c})$ is the determinant of the $3\times 3$-matrix 
$(\bm{a}, \bm{b}, \bm{c})$. 

Let $D$ be a domain in the $(u, v)$-plane. A surface $X\colon D \to \L^3$ is called {\it space-like} if the induced metric, that is, the first fundamental form $\mathrm{I}_{X}$ of $X$ on the surface,  
is positive definite, and from now on we assume a surface is space-like. Set  
\begin{equation}\label{eq-EFG}
E= \langle X_{u}, X_{u} \rangle_{L}, \quad F= \langle X_{u}, X_{v} \rangle_{L}, \quad G= \langle X_{v}, X_{v} \rangle_{L}. 
\end{equation}
The {\it first fundamental form} is then 
\begin{equation}\label{eq-first-def1}
\mathrm{I}_{X} = E\,du^2 +2F\,dudv +G\,dv^2,  
\end{equation}
and the following three statements are equivalent: 
\begin{enumerate}
\item[(a)] $X\colon D \to \L^3$ is a space-like surface.  
\item[(b)] $(EG-F^2)(p) >0$ for any point $p=(u, v)\in D$.  
\item[(c)] $X_{u} \times_{L} X_{v}$ is time-like on $D$, that is, $\langle X_{u} \times_{L} X_{v}, X_{u} \times_{L} X_{v} \rangle_{L} <0$ for any point $p=(u, v) \in D$. 
\end{enumerate}

A surface $X\colon D \to \L^3$ is {\it complete} if $\mathrm{I}_{X}$ is complete.  

Set
\begin{equation}\label{eq-normal-def}
\bm{n}= - \dfrac{X_{u} \times_{L} X_{v}}{\sqrt{-\langle X_{u} \times_{L} X_{v}, X_{u} \times_{L} X_{v} \rangle_{L}}}  
\end{equation} 
to be the unit {\it normal vector field} of $X$. Define
\begin{equation}\label{eq-LMN}
L= -\langle X_{uu}, \bm{n} \rangle_{L}, \quad M= - \langle X_{uv}, \bm{n} \rangle_{L}, \quad N= -\langle X_{vv}, \bm{n} \rangle_{L}. 
\end{equation}
The quadratic form 
\begin{equation}\label{def-second-def}
\mathrm{I}\hspace{-1.2pt}\mathrm{I}_{X} = L\,du^2 +2M\,dudv+N\,dv^2
\end{equation}
is called the {\it second fundamental form} of $X$. 

Let $\kappa_{1}$ and $\kappa_{2}$ be the principal curvatures of $X$, that is, the eigenvalues of the Weingarten matrix 
\[
A:= \begin{pmatrix}
E & F \\
F & G \\
\end{pmatrix}^{-1}
\begin{pmatrix}
L & M \\
M & N \\
\end{pmatrix}. 
\]
Then $\kappa_{1}$ and $\kappa_{2}$ are real-valued functions on $D$. When the two principal curvatures are equal, we say that such a point on the surface is 
an {\it umbilic point}. Moreover, the {\it Gaussian curvature} $K$ and the {\it mean curvature} $H$ of $X$ are written as 
\[ 
K = -\kappa_{1}\kappa_{2} = -\dfrac{LN-M^2}{EG-F^2}, \quad H = \dfrac{\kappa_{1}+\kappa_{2}}{2} = \dfrac{EN-2FM+GL}{2(EG-F^2)}.
\]
This definition of the Gaussian curvature is chosen so as to be consistent with Gauss's Theorema Egregium. 

We now give two typical examples of complete space-like surfaces in $\L^3$ with non-zero CMC.  

\begin{example}[hyperbolic cylinder]\label{ex-hcyl}
The surface $X\colon \R^2 \to \L^3$ given by 
\begin{equation}\label{eq-hyp-cyl}
X(u, v) = (r\sinh{u}, rv, r\cosh{u}), \quad (u, v) \in \R^2, \quad r>0
\end{equation}
is called a {\it hyperbolic cylinder}. Then the unit normal vector field is 
\[
\bm{n}(u, v) = (\sinh{u}, 0, \cosh{u})
\]
and the fundamental forms are 
\begin{equation}\label{eq-hcyl-isoth}
\mathrm{I}_{X} = r^{2}\,(du^2 + dv^2), \quad \mathrm{I}\hspace{-1.2pt}\mathrm{I}_{X} = r\,du^2 
\end{equation}
with 
\[
K = 0, \quad H=\dfrac{1}{2r}. 
\]
\end{example}

\begin{example}[pseudosphere]\label{ex-pseu}
The surface $X\colon \R^2 \to \L^3$ given by 
\begin{equation}\label{eq-pseudo}
X(u, v) = (r\sinh{u}\cos{v}, r\sinh{u}\sin{v}, r\cosh{u}), \quad (u, v) \in \R^2, \quad r>0
\end{equation}
is called a {\it pseudosphere}, with  
\[
\bm{n}(u, v) = (\sinh{u}\cos{v}, \sinh{u}\sin{v}, \cosh{u}),  
\]
\[
\mathrm{I}_{X} = r^{2}\,du^2 + r^{2}\sinh^{2}{u}\,dv^2, \quad \mathrm{I}\hspace{-1.2pt}\mathrm{I}_{X} = r\,du^2 + r\sinh^{2}{u}\,dv^2, \quad K = -\dfrac{1}{r^2}, \quad H=\dfrac{1}{r}.  
\]
\end{example}

\begin{figure}[t]
    \centering
    \begin{minipage}{0.45\textwidth}
        \centering
        \includegraphics[scale=0.6]{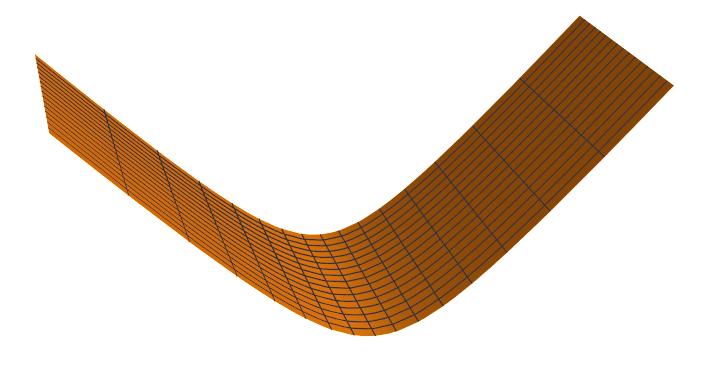}
        \caption{hyperbolic cylinder}
    \end{minipage}
    \hspace{0.05\textwidth}
    \begin{minipage}{0.45\textwidth}
        \centering
        \includegraphics[scale=0.5]{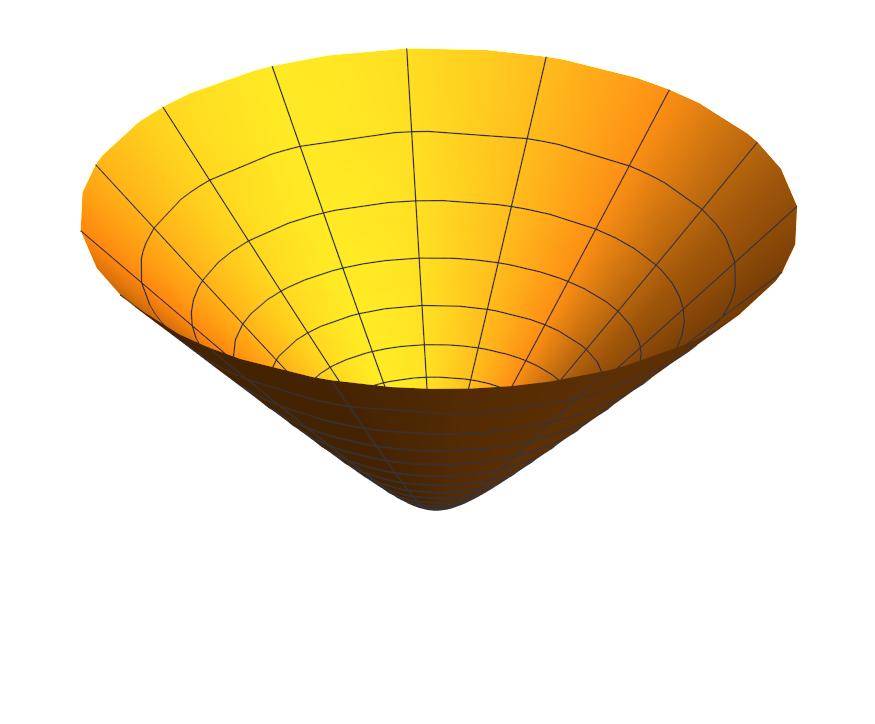}
        \caption{pseudosphere}
    \end{minipage}
\end{figure}

We next explain the integrability conditions of space-like surfaces in $\L^3$. Here, we assume that $i, j, k, \ldots$ are indices varying from $1$ to $2$, 
and that $u_{1}=u$ and $u_{2}=v$. We set 
\[ 
X_{i}= \dfrac{\partial X}{\partial u_{i}}, \quad X_{ij}= \dfrac{\partial^2 X}{\partial u_{i} \partial u_{j}},  \quad \bm{n}_{i} = \dfrac{\partial \bm{n}}{\partial u_{i}}, \quad  
\begin{pmatrix}
g_{11} & g_{12} \\
g_{21} & g_{22} \\
\end{pmatrix} = \begin{pmatrix}
E & F \\
F & G \\
\end{pmatrix}, \quad  \begin{pmatrix}
h_{11} & h_{12} \\
h_{21} & h_{22} \\
\end{pmatrix} = \begin{pmatrix}
L & M \\
M & N \\
\end{pmatrix}
\]
and $(g^{ij}) =(g_{ij})^{-1}$. With Christoffel symbols 
\[
\Gamma^{k}_{ij} := \dfrac{1}{2} \sum_{l} g^{kl} \left( \dfrac{\partial g_{li}}{\partial u_{j}} + \dfrac{\partial g_{lj}}{\partial u_{i}} - \dfrac{\partial g_{ij}}{\partial u_{l}} \right),  
\]
the Gauss and Codazzi equations are  
\begin{eqnarray}
\sum_{t}\, g_{tm} \left( \dfrac{\partial \Gamma^{t}_{ij}}{\partial u_{k}} - \dfrac{\partial \Gamma^{t}_{ik}}{\partial u_{j}} + \sum_{s} \left(\Gamma^{s}_{ij} \Gamma^{t}_{sk} - \Gamma^{s}_{ik} \Gamma^{t}_{sj} \right) \right) &=& - (h_{ij}h_{km} -h_{ik}h_{jm}),  \label{eq-GaussA} \\
\dfrac{\partial h_{ij}}{\partial u_{k}} - \dfrac{\partial h_{ik}}{\partial u_{j}} + \sum_{t} \left( \Gamma^{t}_{ij} h_{tk} - \Gamma^{t}_{ik} h_{tj} \right) &=& 0 \label{eq-CodazziA}, 
\end{eqnarray}
respectively. For a space-like surface in $\L^3$, the following fundamental theorem of surface theory holds. 
 
\begin{fact}\label{fact-fundame}
Let $g_{ij}$ be the coefficients of a Riemannian metric on a simply connected domain $D$, and let $h_{ij}\, (1\leq i, j\leq 2)$ be smooth functions with $h_{ij} =h_{ji}$ on $D$ 
satisfying the differential equations \eqref{eq-GaussA} and \eqref{eq-CodazziA}. Then there exists a space-like surface $X\colon D \to \L^3$ such that 
\[
\mathrm{I}_{X} = \sum_{i, j} g_{ij} du_{i}du_{j}, \quad  \mathrm{I}\hspace{-1.2pt}\mathrm{I}_{X} = \sum_{i, j} h_{ij} du_{i}du_{j}. 
\] 
Furthermore, such a surface $X(D)$ is unique up to isometries of $\L^3$.  
\end{fact} 

\subsection{Moving frames}\label{sec2-2}
Let $X\colon D \to \L^{3}$ be a space-like surface. We take a triple of vector fields $\{ \bm{e}_{1}, \bm{e}_{2}, \bm{e}_{3} \}$ of $\L^3$ which is a Lorentz orthonormal basis, that is,   
\[ 
\langle \bm{e}_{i}, \bm{e}_{j} \rangle_{L} = \delta_{ij}, \quad \langle \bm{e}_{i}, \bm{e}_{3}  \rangle_{L} = 0, \quad  \langle \bm{e}_{3}, \bm{e}_{3}  \rangle_{L} = -1 \quad (i, j =1, 2), 
\]
and assume that $\bm{e}_{1}$ and $\bm{e}_{2}$ are tangent, and $\bm{e}_{3}$ is perpendicular, to the surface. We call this triple a {\it moving frame} of $X$.  

We take the exterior derivative  
\begin{equation}\label{eq-exderi}
dX = \omega_{1} \bm{e}_{1} + \omega_{2} \bm{e}_{2}, 
\end{equation}
where $\omega_{i}\, (i=1, 2)$ are $1$-forms on $D$, which form the dual basis of $\bm{e}_{i}$ considered as tangent vector fields on $D$, 
and  
$
\mathrm{I}_{X} = \omega_{1}^{2} + \omega_{2}^{2}. 
$

The exterior derivative of $\bm{e}_{A}\, (A=1, 2, 3)$  
\begin{equation}\label{eq-deriAB}
d\bm{e}_{A} = \displaystyle \sum_{B} \omega_{AB}\, \bm{e}_{B}  
\end{equation} 
satisfies, since $\{ \bm{e}_{1}, \bm{e}_{2}, \bm{e}_{3} \}$ is orthonormal,  
\[
\begin{cases}
    \omega_{AB} = \omega_{BA}, & (A, B) = (1, 3), (2, 3), (3, 1), (3,2), \\
    \omega_{AB} = -\omega_{BA},        & \text{otherwise}. 
\end{cases}
\]
Restricting $\omega_{AB}$ to $X(D)$, we can represent $\omega_{i3}$ as 
\begin{equation}\label{eq-ai3}
\omega_{i3} = \displaystyle \sum_{j=1}^{2} b_{ij}\, \omega_{j}, \quad (i=1, 2). 
\end{equation}

Taking the exterior derivative of \eqref{eq-exderi} and using $d(dX)=0$, we have
\begin{equation}\label{eq-closed1}
d\omega_{i} = \displaystyle \sum_{j=1}^{2} \omega_{ij}\wedge \omega_{j}, \quad \sum_{k=1}^{2} \omega_{k} \wedge \omega_{k3} =0. 
\end{equation}

The Gaussian curvature $K$ satisfies 
\begin{equation}\label{eq-connection}
d\omega_{12} = -K \omega_{1} \wedge \omega_{2},  
\end{equation} 
and the second fundamental form and mean curvature are  
\[
\mathrm{I}\hspace{-1.2pt}\mathrm{I}_{X} = b_{11} \omega_{1}^{2} + 2b_{12} \omega_{1} \omega_{2} + b_{22} \omega_{2}^{2}, \quad  H = (b_{11}+b_{22})/2.
\]

Taking the exterior derivative of \eqref{eq-deriAB},  we have 
\begin{eqnarray}
d\omega_{12} &=& \omega_{13} \wedge \omega_{32}, \label{eq-121332} \\
d\omega_{13} &=& \omega_{12} \wedge \omega_{23}, \quad d\omega_{23} = \omega_{21} \wedge \omega_{13}. \label{eq-131223}  
\end{eqnarray}
Then \eqref{eq-121332} is equivalent to 
$
K = - (b_{11}b_{22} -b_{12}^{2})
$ 
by \eqref{eq-ai3} and \eqref{eq-connection}. This is the Gauss equation. 

We assume that each point of $X(D)$ is not umbilic. Then the principal curvatures $\kappa_{1}$ and $\kappa_{2}$ differ from each other at each point of $X(D)$. We may assume $\kappa_{1}> \kappa_{2}$. Let the tangent vector fields $\bm{e}_{1}$ and $\bm{e}_{2}$ be the principal curvature vectors at each point. With respect to $\bm{e}_{1}$ and $\bm{e}_{2}$, the symmetric matrix $(b_{ij})$ is diagonalized. Thus we obtain $\omega_{13} = \kappa_{1} \omega_{1}$ and $\omega_{23} =\kappa_{2} \omega_{2}$. Then \eqref{eq-131223} is equivalent to
\begin{eqnarray}
d\kappa_{1} \wedge \omega_{1} + (\kappa_{1}-\kappa_{2})\, \omega_{12}\wedge \omega_{2} &=& 0, \label{eq_Codazzi1} \\
d\kappa_{2} \wedge \omega_{2} + (\kappa_{1}-\kappa_{2})\, \omega_{12}\wedge \omega_{1} &=& 0. \label{eq_Codazzi2}
\end{eqnarray}
This is the Codazzi equation for an umbilic-free space-like surface, in the theory of moving frames. 

\subsection{Isothermal coordinate systems}\label{sec2-3}
Let $X\colon D \to \L^{3}$ be a space-like surface. With suitable choice of parametrization, the first fundamental form becomes  
\begin{equation}\label{eq-isoth-01}
\mathrm{I}_{X} = \lambda^{2}(u, v) (du^2 +dv^2)
\end{equation} 
for some positive smooth function $\lambda (u, v)$, and we call the coordinates $(u, v)$ an {\it isothermal coordinate system}. 
Since $X_{uuv} =X_{uvu},\, X_{vvu}=X_{vuv}$, the Gauss and Codazzi equations are  
\begin{equation}\label{eq_Gauss_iso}
-\dfrac{1}{\lambda^2} \left( \dfrac{\partial^2}{\partial u^2} + \dfrac{\partial^2}{\partial v^2} \right) \log{\lambda} = -\dfrac{LN-M^2}{\lambda^{4}},  
\end{equation} 
\begin{equation}\label{eq-Codazzi-iso}
\dfrac{\partial L}{\partial v} - \dfrac{\partial M}{\partial u} - \dfrac{\partial \log{\lambda}}{\partial v} (L+N) = 0, \quad \dfrac{\partial M}{\partial v} - \dfrac{\partial N}{\partial u} + \dfrac{\partial \log{\lambda}}{\partial u} (L+N) = 0. 
\end{equation} 
As $H = (L+N)/2{\lambda}^2$, this implies
\begin{equation}\label{eq-mean-curv-iso}
\dfrac{1}{2} \dfrac{\partial (L-N)}{\partial v} - \dfrac{\partial M}{\partial u} = -\lambda^{2} \frac{\partial H}{\partial v}, \quad \dfrac{1}{2} \dfrac{\partial (L-N)}{\partial u} + \dfrac{\partial M}{\partial v} = \lambda^{2} \frac{\partial H}{\partial u}. 
\end{equation}

We set $z= u+\mathrm{i} v$ and consider the domain $D$ of $X$ as a subset of $\C$. We define a complex-valued function on $D$ by 
\begin{equation}\label{eq-Hopf-comp}
\phi (z) := \dfrac{L-N}{2} - \mathrm{i} M.  
\end{equation}
The complex quadratic differential $\phi (z) dz^2$ is called the {\it Hopf differential} of $X$. By \eqref{eq-mean-curv-iso}, $\phi (z)$ is holomorphic if and only if $H$ is constant. 
Then we can obtain 
\begin{equation}\label{eq_Hopf-diff}
|\phi (z)| = \dfrac{\lambda^2}{2}\,|\kappa_{1} -\kappa_{2}|. 
\end{equation}
Thus a space-like surface in $\L^3$ with CMC is either entirely umbilic or has isolated umbilic points, because $\phi (z)$ is holomorphic. 

\section{Main results}\label{sec3}
We first recall the definition of isothermic coordinate systems. 
Let $X\colon D \to \L^{3}$ be a space-like surface. The coordinates $(u, v)$ are called a {\it curvature line coordinate system} if the second fundamental form of $X$ satisfies 
\begin{equation}\label{eq-isoth-02}
\mathrm{I}\hspace{-1.2pt}\mathrm{I}_{X} = L\,du^2 +N\,dv^2. 
\end{equation}
When the coordinates $(u, v)$ satisfy both \eqref{eq-isoth-01} and \eqref{eq-isoth-02}, 
$(u, v)$ are called an {\it isothermic coordinate system}.  

\begin{example}\label{ex-hyper-ccc}
For the hyperbolic cylinder defined by \eqref{eq-hyp-cyl},  
by \eqref{eq-hcyl-isoth} the coordinates $(u, v)$ are an isothermic coordinate system.  
\end{example}

We next show the existence of an isothermic coordinate system for umbilic-free space-like surfaces in $\L^3$ with non-zero CMC. 

\begin{theorem}\label{thm-ICC}
For a simply-connected domain $D$ and a space-like surface $X\colon D \to \L^3$ with CMC-$H\, (H>0)$,  
assume that every $p\in D$ is non-umbilic.  
Then there exists an isothermic coordinate system $(u, v)$ on $D$ satisfying  
\begin{equation}\label{eq-12}
\mathrm{I}_{X} = e^{\omega} (du^2 +dv^2), \quad \mathrm{I}\hspace{-1.2pt}\mathrm{I}_{X} = 2He^{\frac{\omega}{2}} \left( \cosh\frac{\omega}{2} du^2 + \sinh\frac{\omega}{2} dv^2 \right), 
\end{equation}
where $\omega =\omega (u, v)$ satisfies 
\begin{equation}\label{eq-LSG}
\omega_{uu} + \omega_{vv} -4H^2 \sinh \omega = 0. 
\end{equation}

Conversely, for given $H>0$ and solution $\omega (u, v)$ of \eqref{eq-LSG}, 
there exists a space-like surface with CMC-$H$ whose fundamental forms are given by \eqref{eq-12}, which is unique 
up to isometries of $\L^3$.   
\end{theorem}

\begin{proof} 
Assume $\kappa_{1}$ and $\kappa_{2}$ of $X$ satisfy $\kappa_{1} > \kappa_{2}$ on $D$ and 
$\bm{e}_{i}\, (i=1, 2)$ is a unit vector field determining the principal direction corresponding to $\kappa_{i}$. 
Then $\bm{e}_{1}$ and $\bm{e}_{2}$ are orthogonal. Let $\omega_{i}$ be the dual $1$-form of $\bm{e}_{i}$. By \eqref{eq-closed1}, \eqref{eq_Codazzi1} and
\eqref{eq_Codazzi2}, we obtain
\begin{equation}\label{eq-omega} 
d\omega_{1} = \omega_{12} \wedge \omega_{2} = -\dfrac{1}{\kappa_{1}-\kappa_{2}} (d\kappa_{1}\wedge \omega_{1}), \quad  d\omega_{2} = \omega_{21} \wedge \omega_{1} = \dfrac{1}{\kappa_{1}-\kappa_{2}} (d\kappa_{2}\wedge \omega_{2}). 
\end{equation}
We consider the $1$-form $\theta_{i}$ defined by 
\[
\theta_{i}  = \sqrt{\dfrac{\kappa_{1}-\kappa_{2}}{2}}\,\omega_{i}, \quad (i=1, 2). 
\]
By \eqref{eq-omega}, we have  
\begin{equation}\label{eq-theta1}
d\theta_{1} = d\sqrt{\kappa_{1} -H} \wedge \omega_{1} - \dfrac{1}{2\sqrt{\kappa_{1}-H}}\, (d\kappa_{1} \wedge \omega_{1}).  
\end{equation}
On the other hand,  we obtain 
\[
d\sqrt{\kappa_{1} -H} = \dfrac{1}{2\sqrt{\kappa_{1}-H}}\, d\kappa_{1}, 
\]
because $H (=(\kappa_{1}+\kappa_{2})/2)$ is constant. Thus $\theta_{1}$ is closed. Similarly, $\theta_{2}$ is closed. 
By the Poincar\'e lemma, there exists a function $u_{i}$ on $D$ such that $\theta_{i} =du_{i}\, (i=1, 2)$. 
If we choose $\tau$ so that  
\[
e^{-\tau} = \dfrac{\kappa_{1}-\kappa_{2}}{2}, 
\] 
we have $\omega_{i} =e^{(\tau /2)} du_{i}\, (i=1, 2)$ and 
\[
\mathrm{I}_{X} = e^{\tau} (du_{1}^{2}+du_{2}^{2}), \quad \mathrm{I}\hspace{-1.2pt}\mathrm{I}_{X} = (He^{\tau} +1) du_{1}^{2} +  (He^{\tau} -1) du_{2}^{2}. 
\]
Moreover, if we choose $\omega, u, v$ so that $e^{\tau} =e^{\omega}/H,\, u_{1} =\sqrt{H}u, u_{2} =\sqrt{H}v$, we obtain \eqref{eq-12} and  \eqref{eq-LSG}. 

The converse follows by Fact \ref{fact-fundame}, because the Codazzi equation holds automatically for \eqref{eq-12}, and 
\eqref{eq-LSG} is the Gauss equation \eqref{eq_Gauss_iso}. 
\end{proof}

We will apply Theorem \ref{thm-ICC} to complete space-like CMC surfaces in $\L^3$. 
Let $\Sigma$ be a connected and oriented $2$-manifold. 
It is well-known that every complete space-like surface $X\colon \Sigma \to \L^3$ is entire (e.g. see \cite[Proposition 9.3]{AMR2016}). 
In particular, $\Sigma$ is conformally equivalent to either the entire plane $\R^2$ or the unit disk $\D$, because $\Sigma$ is simply-connected and non-compact. 
On the other hand, Yamada \cite[Proposition 3.1]{Yam1988} proved the following result. 

\begin{fact}\label{fa-Yam}
The only solution of \eqref{eq-LSG} on the entire plane $\R^2$ is the trivial solution $\omega \equiv 0$. 
\end{fact}

\begin{remark}\label{rmk-Nagoya}
Consider the ordinary differential equation
\begin{equation}\label{eq-Nagoya}
\omega ''(r) +\dfrac{1}{r} \omega' (r) = 4H^2 \sinh \omega (r) \quad \mathrm{for} \quad r\geq 0, 
\end{equation}
where $'\,:= d/dr$. Then the solution of \eqref{eq-Nagoya} is a radially symmetric solution of \eqref{eq-LSG} with $r=\sqrt{u^2 +v^2}$. 
The equation \eqref{eq-Nagoya} is a Painlev\'e equation of the third kind. 
\end{remark}

By combining Theorem \ref{thm-ICC} and Fact \ref{fa-Yam}, we can show the following uniqueness theorem.  

\begin{corollary}\label{cor-uni}
Any complete non-umbilic space-like surface in $\L^3$ with non-zero CMC-$H$ defined on the entire plane $\R^2$ must be a hyperbolic cylinder. 
\end{corollary}  

\begin{proof} 
By reversing the unit normal vector if necessary, we may assume $H>0$. Moreover, we may assume $H=1/2$ without loss of generality, because a homothety of the surface 
scales the mean curvature accordingly. Since the surface is non-umbilic and simply-connected, we can take an isothermic coordinate system $(u, v)$ as in Theorem \ref{thm-ICC} 
on the entire plane $\R^2$. 
By Fact \ref{fa-Yam}, we have $\omega \equiv 0$. Thus 
\[
\mathrm{I}_{X} = du^2 +dv^2, \quad \mathrm{I}\hspace{-1.2pt}\mathrm{I}_{X} = du^2 
\]
and the corresponding space-like surface in $\L^3$ is a hyperbolic cylinder 
by Example \ref{ex-hcyl} and Fact \ref{fact-fundame}. 
\end{proof}

Milnor \cite[Theorem 8]{Mi1983} and Yamada \cite[Theorem]{Yam1988} proved the following uniqueness theorem: 

\begin{theorem}\label{thm-KY}
Any complete space-like surface in $\L^3$ with non-zero CMC and principal curvatures $\kappa_{1}$ and $\kappa_{2}$ satisfying  
\begin{equation}\label{eq-eps}
(\kappa_{1} -\kappa_{2})^{2} \geq \varepsilon^{2}
\end{equation} 
for some positive number $\varepsilon$ must be a hyperbolic cylinder. 
\end{theorem}

The assumption \eqref{eq-eps} is necessary. In fact, there exist non-umbilic space-like surfaces defined on $\D$ in $\L^3$ with non-zero CMC  
on which 
\[
\displaystyle \lim_{(u, v) \to \partial \D} (\kappa_{1} -\kappa_{2})^{2} =0,  
\]
where $\partial \D$ is the boundary of $\D$ (\cite[Remark 1]{Yam1988}). 

\begin{proof}  
By Corollary \ref{cor-uni}, we only need to consider $\Sigma =\D$. The following argument is based on \cite[Section 4]{Yam1988}. 
Suppose that there exists a complete space-like surface defined on $\D$ in $\L^3$ with non-zero CMC  
and principal curvatures $\kappa_{1}, \kappa_{2}$ satisfying \eqref{eq-eps}. Then we can take an isothermic coordinate system $(u, v)$ in Theorem \ref{thm-ICC} on $\D$. 
Since $\mathrm{I}_{X}$ is complete, we have 
\[
\displaystyle \lim_{(u, v) \to \partial \D} e^{\omega} = +\infty.  
\]
From \eqref{eq_Hopf-diff} and \eqref{eq-eps}, $\phi$ is a non-vanishing holomorphic function on $\D$ which satisfies 
\[
\displaystyle \lim_{(u, v) \to \partial \D} |\phi (z)| = +\infty. 
\]

Set $\psi =1/\phi$. Then $\psi$ is holomorphic on $\D$ and continuous on $\D \cup \partial \D$ with $\psi |_{\partial \D} =0$. By the Cauchy expression, we obtain 
\[
\psi (0) = \dfrac{1}{2\pi\mathrm{i}} \int_{\partial \D} \dfrac{\psi (z)}{z}\, dz =0. 
\]
This contradicts that $\phi$ is a non-vanishing holomorphic function on $\D$.  
\end{proof}

Finally, we give a corollary of Theorem \ref{thm-KY}. 
By using the argument for the existence of an isothermic coordinate system, Klotz and Osserman \cite{KO1966} proved that a complete CMC surface in $\E^3$  
whose Gaussian curvature does not change sign is either a sphere, a minimal surface, or a right circular cylinder. In $\L^3$, as a corollary of Theorem \ref{thm-KY},   
we can show the following uniqueness result.  

\begin{corollary}\label{cor-KO}
If $X\colon \Sigma \to \L^3$ is a complete space-like CMC-$H$ surface with non-negative Gaussian curvature,  
then it is a plane or hyperbolic cylinder. 
\end{corollary}

Before proving this, we call the Calabi-Bernstein theorem, first shown by Calabi \cite{Ca1970}. 
For other proofs of this fact, see \cite{AP2001, Ca1970, Ka2013, Ko1983, Ro1996, UY2006} for example. 
We also note that a space-like surface  in $\L^3$ whose mean curvature vanishes everywhere is called a maximal surface in $\L^3$. 

\begin{fact}[Calabi-Bernstein theorem]\label{fa-CB}
Any complete maximal surface in $\L^3$ must be a plane. 
\end{fact}

\begin{proof}
The proof of Corollary \ref{cor-KO} is straightforward. If $H=0$, then $X(\Sigma)$ must be a plane by Fact \ref{fa-CB}. If $H\not= 0$,  
since $K= -\kappa_{1}\kappa_{2} >0$, we obtain $(\kappa_{1}-\kappa_{2})^2 =4(H^2 +K) \geq 4H^2$. By virtue of Theorem \ref{thm-KY}, 
$X(\Sigma)$ must be  a hyperbolic cylinder. 
\end{proof}

Treibergs \cite{Tr1982} showed there exist many complete space-like surfaces in $\L^3$ with non-zero CMC other than hyperbolic cylinders. 
By Corollary \ref{cor-KO}, those surfaces must have at least one point $p\in \Sigma$ where $K(p)< 0$. 

On the other hand, it is known that the Gaussian curvature of a complete space-like surface in $\L^3$ with non-zero CMC  
must be non-positive (see \cite[Corollary on page 49]{Tr1982}, \cite[Corollary on page 419]{CY1976} for example). 
Thus a complete space-like surfaces in $\L^3$ with non-zero CMC does not have any point $p\in \Sigma$ where $K(p)> 0$, 
and in fact, the conclusion of Corollary \ref{cor-KO} includes only surfaces $H\equiv 0$.

%%%%%%%%%%%% References %%%%%%%%%%%%%
%%
%<Author name> is written as Initial of Given Name, and Family Name.
%<Title> is written in roman letters.
%<Journal name> should be abbreviated according to
% the MR Serials Abbreviations List of Mathematical Reviews:
% (Abbreviations of Names of Serials; http://www.ams.org/mr-database)
%For <Pages>, use en-dash "--" between page numbers.
%%

\end{document}